\documentclass[final]{siamltex}

\title{An abstract criterion on the existence and global stability of stationary solutions for random dynamical systems and its applications\thanks{This work was partially supported by the National Natural Science Foundation of China (NSFC) under
Grants No.12171321, No.11971316, No.11771295, No.11501369 and No.11371252; the NSF of Shanghai Grants under No.25ZR1401278, No.19ZR1437100 and No.20JC1413800; Chen Guang Project (14CG43) of Shanghai Municipal Education Commission, Shanghai Education Development Foundation; Yangfan Program of Shanghai (14YF1409100) and Shanghai Gaofeng Project for University Academic Program Development.}}

\author{Xiang Lv\thanks{Corresponding author. Department of Mathematics, Shanghai Normal
University, Shanghai 200234, PR China ({\tt lvxiang@shnu.edu.cn}).}}

\usepackage[leqno]{amsmath}
\usepackage{amsmath,amssymb,amsfonts}
\usepackage{color,xcolor}
\usepackage{graphicx}
\usepackage{manfnt}
\usepackage[all]{xy}
\usepackage{enumerate}
\usepackage{mathrsfs}

\begin{document}

\maketitle

\begin{abstract}
We prove a concise and easily verifiable criterion on the existence and global stability of stationary solutions for random dynamical systems (RDSs). As a consequence, we can show that the $\omega$-limit sets of all pullback trajectories of semilnear/nonlinear stochastic differential equations (SDEs) with additive/multiplicative white noise are composed of nontrivial random equilibria. The proof is different from the classical RDS scheme, which was established in \cite{CKS}. Furthermore, in the applications of stability analysis for SDEs, our conditions are not only sufficient but indeed sharp.
\end{abstract}



\pagestyle{myheadings}
\thispagestyle{plain}

\section{Introduction}
Due to practical applications, one of the important problems in the study of long-term behaviour for stochastic differential equations (SDEs) is to consider various types of stochastic stability, including exponential stability in mean square, globally asymptotical stability in probability, and almost sure stability, which has been extensively and intensively investigated during past decades, see \cite{H,Ic1,Kha,Ko,Ku,Liu,LM,M,M2}. However, the existing literature primarily focuses on the stability analysis of trivial or constant stationary solutions, both in finite-dimensional systems and infinite-dimensional frameworks. The main purpose of this paper is to establish the existence and global stability of nontrivial stationary solutions for SDEs by applying an abstract fixed point theorem for random dynamical systems (RDSs).

To be specific, in the finite-dimensional setting, Kozin \cite{Ko} established fundamental results for linear stochastic systems, laying a rigorous foundation for subsequent research.  Around the same period, Kushner \cite{Ku} developed the Lyapunov function theory for strong Markov processes with applications to control problems. Building on these works, Kha'sminskii \cite{Kha} made a comprehensive work on the stability theory for solutions of It\^{o} SDEs.

The investigation of stability for infinite-dimensional SDEs can go back by Haussmann \cite{H} for linear systems and Ichikawa \cite{Ic1} for semilinear systems. These pioneering studies inspired significant follow-up work, including contributions by Caraballo and Real \cite{CR},  Leha, Ritter and Maslowski \cite{LRM}, Liu \cite{Liu} and Liu and Mao \cite{LM}. The primary tool employed in the research of stability for constant stationary solutions is the Lyapunov functional method, whose main challenge lies in constructing appropriate Lyapunov functions.

In contrast to constant stationary solutions, we are mainly interested in the asymptotical stability of nontrivial stationary solutions. This issue has attracted limited attention in previous studies, such as \cite{CGS,CKS,JL1,JL2,JL3}. For this problem, the Lyapunov functional method proves inadequate, whereas techniques from RDSs can offer an effective solution approach. In fact, we can obtain much more information if the stochastic equation can induce an RDS, and then stationary solutions represent random attractors consisting of a single point, see \cite{CF,CF2,FGS}. Using this viewpoint, under the assumptions of commutativity on drift and diffusion coefficients, Caraballo, Kloeden and Schmalfu{\ss} \cite{CKS} proved the existence of a unique stationary solution for semilinear stochastic evolution equations. In our recent works \cite{JL1,JL2}, we have demonstrated that the stochastic flow generated by SDEs admits a globally attracting random equilibrium, under the hypothesis that the nonlinear drift function is bounded and satisfies either monotonicity or anti-monotonicity conditions.

In this paper, we shall first establish a concise and easily verifiable criterion (see Theorem \ref{thm1}) for guaranteeing the globally attracting random equilibrium of RDSs. In the applications, we only consider the global stability of  nontrivial stationary solutions for semilnear or nonlinear SDEs with additive or multiplicative white noise, respectively. Departing from the approaches in \cite{CKS,JL1,JL2}, we introduce a completely different technical route and significantly weaken the required assumptions. Actually, the methodology given here can also be extended to consider stochastic functional differential equations (SFDEs), stochastic partial differential equations (SPDEs) and related problems.

The rest of this paper is organized as follows. In Section 2, we introduce some basic notations and present an abstract criterion on the existence and global stability of random equilibria for random dynamical systems, see Theorem \ref{thm1}. In Section 3, we demonstrate the broad applicability of Theorem \ref{thm1} for the global stability of SDEs. In Section 4, we provide concluding remarks and give several open problems for future research.

\section{An abstract criterion}
In this section, we will prove an abstract criterion on the existence and global stability of random equilibria for RDSs. To do this, we need to introduce some basic definitions of RDSs, and then show our main results at the end of this section. The reader is referred to \cite{A,Chu} for more details. Let $(X,d)$ be a complete separable metric space (i.e., Polish space) equipped with the Borel $\sigma$-algebra $\mathscr{B}(X)$ and $(\Omega,\mathscr{F},\mathbb{P})$ be a probability space.

\begin{definition}
$\theta\equiv\bigl(\Omega,\mathscr{F},\mathbb{P},\{\theta_t,t\in\mathbb{R}\}\bigr)$ is called a metric dynamical system (MDS) if
\[\theta:\mathbb{R}\times\Omega\mapsto\Omega,\qquad \theta_0={\rm id},\qquad \theta_{t_2}\circ\theta_{t_1}=\theta_{t_1+t_2}\]
for all $t_1,t_2\in\mathbb{R}$, which is $\bigl(\mathscr{B}(\mathbb{R})\otimes\mathscr{F}, \mathscr{F}\bigr)$-measurable. In addition, we assume that $\theta_t\mathbb{P}=\mathbb{P}$ for all $t\in\mathbb{R}$.
\end{definition}

\begin{definition}\label{DRDS}
An RDS on the state space $X$ with an MDS $\theta$ is a mapping
\[\varphi:\mathbb{R}_+\times\Omega\times X\mapsto X, \quad (t,\omega,x)\mapsto\varphi(t,\omega,x),\]
which is $\bigl(\mathscr{B}(\mathbb{R}_+)\otimes\mathscr{F}\otimes\mathscr{B}(X),
\mathscr{B}(X)\bigr)$-measurable such that for all $\omega\in\Omega$,
\begin{enumerate}[{\rm(i)}]
\item $\varphi(0,\omega,\cdot)$ is the identity on $X$;
\item $\varphi(t_1+t_2,\omega,x)=\varphi\bigl(t_2,\theta_{t_1}\omega,\varphi(t_1,\omega,x)\bigr)$ for all $t_1,t_2\in\mathbb{R}_+$ and $x\in X$;
\item $\varphi(t,\omega,\cdot):X\to X$ is continuous for all $t\in
\mathbb{R}_+$.
\end{enumerate}
\end{definition}

\begin{definition}\label{Tem}
A random variable $R(\omega)$ is said to be tempered with respect to the MDS $\theta\equiv\bigl(\Omega,\mathscr{F},\mathbb{P},\{\theta_t,t\in\mathbb{R}\}\bigr)$ if
\[\sup_{t\geq0}\bigl\{e^{-\gamma t}|R(\theta_{-t}\omega)|\bigr\}<\infty\quad for\ \ all\ \ \omega\in\Omega\ \ and\ \ \gamma>0.\]
Furthermore, a random variable $R(\omega)$ is said to be $\gamma_0$-tempered (weakly tempered) with respect to the MDS $\theta\equiv\bigl(\Omega,\mathscr{F},\mathbb{P},\{\theta_t,t\in\mathbb{R}\}\bigr)$ if
\[\sup_{t\geq0}\bigl\{e^{-\gamma t}|R(\theta_{-t}\omega)|\bigr\}<\infty\quad for\ \ all\ \ \omega\in\Omega\ \ and\ \ \gamma\geq\gamma_0.\]
\end{definition}

\textsc{{\it Remark}} 1.
By Definition \ref{Tem}, a tempered random variable $R(\omega)$ must be $\gamma_0$-tempered for all $\gamma_0>0$. However, the inverse may not be true.

\begin{definition}\label{Equili}
A random variable $u:\Omega\rightarrow X$ is said to be a random equilibrium (or stationary solution) of the RDS $(\theta,\varphi)$ if it is invariant with respect to  $(\theta,\varphi)$, i.e.
\[\varphi(t,\omega)u(\omega)=u(\theta_t\omega)\quad for\ \ all\ \ t\geq0\ \ and\ \ \omega\in\Omega.\]
\end{definition}

In what follows, we can state our main results.
\begin{theorem}\label{thm1} Let $(\theta,\varphi)$ be an RDS with the state space $X$.
Assume that $(\theta,\varphi)$ satisfies the following conditions:
\begin{enumerate}[{\rm(H1)}]
\item There exist a real number $\lambda>0$ and a random variable $\widehat{R}(\omega)$ such that for any $x,y\in X$,
\begin{equation}\label{H1}
d\bigl(\varphi(t,\theta_{-t}\omega,x),\varphi(t,\theta_{-t}\omega,y)\bigr)\leq\widehat{R}(\omega)e^{-\lambda t}d\bigl(x,y\bigr),\ t\geq t_0,\ \omega\in\Omega,
\end{equation}
where $t_0\geq0$ is a positive constant;
\end{enumerate}

\begin{enumerate}[{\rm(H2)}]
\item For any $x\in X$, there exists a $\lambda_0$-tempered random variable $R_x(\omega)$ such that
\begin{equation}\label{H2}
\sup_{t\geq0}d\bigl(\varphi(t,\theta_{-t}\omega,x),x\bigr)\leq R_x(\omega),\ \omega\in\Omega,
\end{equation}
where $0<\lambda_0<\lambda$.
\end{enumerate}
Then there exists a unique random equilibrium $U(\omega)$ of the RDS $(\theta,\varphi)$ such that
\begin{equation}\label{Conclu1}
\lim_{t\rightarrow\infty}\varphi(t,\theta_{-t}\omega,x)=U(\omega)\quad {\rm in}\quad X
\end{equation}
for any $x\in X$ and $\omega\in\Omega$. Moreover, the random equilibrium $U:\Omega\rightarrow X$ is tempered, i.e.,
\begin{equation}\label{Conclu2}
d\bigl(U(\omega),x\bigr)\leq R_x(\omega)
\end{equation}
for any $x\in X$ and $\omega\in\Omega$.
\end{theorem}
\begin{proof}
First, we claim that for any $x\in X$,  $\{\varphi(t,\theta_{-t}\omega,x):t\geq0\}$ is Cauchy in $X$, i.e., there exists a random equilibrium $U_x:\Omega\rightarrow X$ such that
\begin{equation}\label{eq1}
\lim_{t\rightarrow\infty}\varphi(t,\theta_{-t}\omega,x)=U_x(\omega),\quad \omega\in\Omega.
\end{equation}
In fact, using (H1) and (H2), for any $t_0\leq t_1<t_2$ and $\omega\in\Omega$, it is easy to see that
\begin{align}\label{eq2}
&d\bigl(\varphi(t_2,\theta_{-t_2}\omega,x),\varphi(t_1,\theta_{-t_1}\omega,x)\bigr)\nonumber\\
=&d\Bigl(\varphi\bigl(t_1,\theta_{-t_1}\omega,\varphi(t_2-t_1,\theta_{-t_2}\omega,x)\bigr),\varphi(t_1,\theta_{-t_1}\omega,x)\Bigr)\nonumber\\
\leq&\widehat{R}(\omega)e^{-\lambda t_1}d\bigl(\varphi(t_2-t_1,\theta_{-(t_2-t_1)}\circ\theta_{-t_1}\omega,x),x\bigr)\nonumber\\
\leq&\widehat{R}(\omega)e^{-\lambda t_1}\sup_{t\geq0}d\bigl(\varphi(t,\theta_{-t}\circ\theta_{-t_1}\omega,x),x\bigr)\nonumber\\
\leq&\widehat{R}(\omega)e^{-\lambda t_1}R_x(\theta_{-t_1}\omega)\nonumber\\
=&e^{-(\lambda-\lambda_0)t_1}\widehat{R}(\omega)e^{-\lambda_0 t_1}R_x(\theta_{-t_1}\omega)\nonumber\\
\leq&e^{-(\lambda-\lambda_0)t_1}\widehat{R}(\omega)\sup_{t\geq0}\left\{e^{-\lambda_0 t}R_x(\theta_{-t}\omega)\right\}\nonumber\\
&\longrightarrow0\quad {\rm as}\quad t_1\rightarrow\infty,
\end{align}
which implies that (\ref{eq1}) holds. Besides this, we can also have that
\begin{align}\label{eq3}
&\varphi\bigl(t,\omega,U_x(\omega)\bigr)\nonumber\\
=&\varphi\bigl(t,\omega,\lim_{\tilde{t}\rightarrow\infty}\varphi(\tilde{t},\theta_{-\tilde{t}}\omega,x)\bigr)\qquad\qquad\qquad\,{\rm by\ (\ref{eq1})}\nonumber\\
=&\lim_{\tilde{t}\rightarrow\infty}\varphi\bigl(t,\omega,\varphi(\tilde{t},\theta_{-\tilde{t}}\omega,x)\bigr)\qquad\qquad\qquad{\rm by\ continuity}\nonumber\\
=&\lim_{\tilde{t}\rightarrow\infty}\varphi(t+\tilde{t},\theta_{-\tilde{t}}\omega,x)\qquad\qquad\qquad\qquad{\rm by\ cocycle}\nonumber\\
=&\lim_{\tilde{t}\rightarrow\infty}\varphi(t+\tilde{t},\theta_{-(t+\tilde{t})}\circ\theta_t\omega,x)\nonumber\\
=&U_x(\theta_t\omega),\quad t\geq0,\ \omega\in\Omega.
\end{align}
That is, $U_x(\omega)$ is a random equilibrium. Furthermore, by (\ref{H2}) and (\ref{eq1}), it is clear that
\begin{equation}\label{eq4}
d\bigl(U_x(\omega),x\bigr)\leq R_x(\omega)
\end{equation}
for all $x\in X$ and $\omega\in\Omega$. Finally, we will show that for any $x,y\in X$,
\begin{equation}\label{eq5}
U_x(\omega)=U_y(\omega),\quad \omega\in\Omega.
\end{equation}
Since $U_x(\omega)$ and $U_y(\omega)$ are two random equilibria of $(\theta,\varphi)$, it follows that
\begin{equation}\label{eq6}\varphi\bigl(t,\theta_{-t}\omega,U_x(\theta_{-t}\omega)\bigr)=U_x(\omega),\quad t\geq0\end{equation}
and
\begin{equation}\label{eq7}\varphi\bigl(t,\theta_{-t}\omega,U_y(\theta_{-t}\omega)\bigr)=U_y(\omega),\quad t\geq0\end{equation}
for all $\omega\in\Omega$. Combining (\ref{H1}), (\ref{H2}), (\ref{eq4}), (\ref{eq6}) and (\ref{eq7}), we can easily get that
\begin{align}\label{eq8}
&d\bigl(U_x(\omega),U_y(\omega)\bigr)\nonumber\\
=&d\Bigl(\varphi\bigl(t,\theta_{-t}\omega,U_x(\theta_{-t}\omega)\bigr),\varphi\bigl(t,\theta_{-t}\omega,U_y(\theta_{-t}\omega)\bigr)\Bigr)\nonumber\\
\leq&\widehat{R}(\omega)e^{-\lambda t}d\bigl(U_x(\theta_{-t}\omega),U_y(\theta_{-t}\omega)\bigr)\nonumber\\
\leq&\widehat{R}(\omega)e^{-\lambda t}\left[d\bigl(U_x(\theta_{-t}\omega),x\bigr)+d\bigl(U_y(\theta_{-t}\omega),y\bigr)+d(x,y)\right]\nonumber\\
\leq&\widehat{R}(\omega)e^{-\lambda t}\left[R_x(\theta_{-t}\omega)+R_y(\theta_{-t}\omega)+d(x,y)\right]\nonumber\\
=&e^{-(\lambda-\lambda_0)t}\widehat{R}(\omega)\left[e^{-\lambda_0 t}R_x(\theta_{-t}\omega)+e^{-\lambda_0 t}R_y(\theta_{-t}\omega)+e^{-\lambda_0 t}d(x,y)\right]\nonumber\\
\leq&e^{-(\lambda-\lambda_0)t}\widehat{R}(\omega)\left(\sup_{t\geq0}\bigl\{e^{-\lambda_0 t}R_x(\theta_{-t}\omega)\bigr\}+\sup_{t\geq0}\bigl\{e^{-\lambda_0 t}R_y(\theta_{-t}\omega)\bigr\}+d(x,y)\right)
\end{align}
where $t\geq t_0$ and $\omega\in\Omega$. Let $t\rightarrow\infty$ in (\ref{eq8}), it is immediate that (\ref{eq5}) is true, which ends the proof.
\end{proof}

\textsc{{\it Remark}} 2.
If there exists a $x_0\in X$ such that
$\varphi(t,\omega,x_0)=x_0$
for all $\omega\in\Omega$ and $t\geq0$, then $x_0$ is a globally stable random equilibrium.

\textsc{{\it Remark}} 3.
In fact, the constant $\lambda_0$ given in (H2) can depend on the initial value $x$.

\textsc{{\it Remark}} 4. For the case of discrete time, assumptions (H1) and (H2) can be replaced by the following version (H1') and (H2')
\begin{enumerate}[{\rm(H1')}]
\item There exist a real number $\lambda>0$ and a random variable $\widehat R(\omega)$ such that for any $x,y\in X$,
\[
d\bigl(\varphi(n,\theta_{-n}\omega,x),\varphi(n,\theta_{-n}\omega,y)\bigr)\leq\widehat{R}(\omega)e^{-\lambda n}d\bigl(x,y\bigr),\ n\geq n_0,\ \omega\in\Omega,
\]
where $n_0\geq1$ is a positive integer;
\end{enumerate}

\begin{enumerate}[{\rm(H2')}]
\item For any $x\in X$, there exists a $\lambda_0$-tempered random variable $R_x(\omega)$ such that
\[
\sup_{n\in\mathbb{N}}d\bigl(\varphi(n,\theta_{-n}\omega,x),x\bigr)\leq R_x(\omega),\ \omega\in\Omega,
\]
where $0<\lambda_0<\lambda$.
\end{enumerate}

\section{Applications}
In this section, we will use Theorem \ref{thm1} to consider the global stability of random equilibria for different stochastic systems, which will be divided into four parts. For tractability, we consider only the case of SDEs. Nevertheless, the core methodology can be generalized naturally to SFDEs, SPDEs and so on.

\subsection{\bfseries Additive white noise: Globally Lipschitz condition}
Firstly, we consider the following $n$-dimensional SDEs with additive white noise
\begin{equation}\label{SDE1-1}
  dx(t)=\Bigl[Ax(t)+f\bigl(x(t)\bigr)\Bigr]dt+\Sigma dB(t),
\end{equation}
where $B(t)=\left(B_1(t),\ldots,B_m(t)\right)^{\rm T}$ is an $m$-dimensional two-sided Brownian motion on the standard Wiener space $(\Omega,\mathscr{F},\mathbb{P})$. Here, $\mathscr{F}$ is the Borel $\sigma$-algebra of
$\Omega=C_0(\mathbb{R},\mathbb{R}^m)=\{\omega(t)\ {\rm is} \ \mbox{continuous},\ \omega(0)=0,\ t\in\mathbb{R}\}$. In addition, $A=(A_{ij})_{n\times n}$ is an $n\times n$-dimensional matrix, $f:\mathbb{R}^n\rightarrow\mathbb{R}^n$ and $\Sigma=(\Sigma_{ij})_{n\times m}$ is an $n\times m$-dimensional matrix.

From now on, we set the Euclidean norm $|x|:=(\sum_{i=1}^n|x_i|^2)^{\frac12}$ and $\|D\|:=(\sum_{i=1}^n\sum_{j=1}^m|D_{ij}|^2)^{\frac12}$, where $x\in\mathbb{R}^n$ and $D\in\mathbb{R}^{n\times m}$. Let us denote by $\theta$ the Wiener shift operator defined by $\theta_t\omega(\cdot)=\omega(t+\cdot)-\omega(t)$ for all $t\in\mathbb{R}$, which is an ergodic MDS. To our purpose, we will present some hypotheses on $A$ and $f$:
\begin{enumerate}[({A1})]
\item  The top Lyapunov exponent of $\Phi(t)=e^{At}$ is a negative real number. That is, there exist two positive constants $\lambda>0$ and $C>0$ such that
\begin{equation}\label{eq10}
\|\Phi(t)\|:=\left(\sum_{i=1}^n\sum_{j=1}^n|\Phi_{ij}(t)|^2\right)^{\frac12}\leq Ce^{-\lambda t}
\end{equation}
for all $t\geq0$.
\end{enumerate}

\begin{enumerate}[({A2})]
\item $f$ is globally Lipschitz continuous, i.e.,
\begin{equation}\label{eq11}
|f(x)-f(y)|\leq L|x-y|
\end{equation}
for all $x,y\in\mathbb{R}^n$, where $L>0$ is the Lipschitz constant satisfying $\frac{LC}{\lambda}<1$.
\end{enumerate}

Define $\varphi_1(t,\omega,x)=x(t,\omega,x)$ to be the unique solution of (\ref{SDE1-1}) with the initial value $x(0)=x\in\mathbb{R}^n$, it is well known that $(\theta,\varphi_1)$ is an RDS generated by (\ref{SDE1-1}). In what follows, we will show the asymptotic behavior of $\varphi_1$.

\begin{theorem}\label{thm2} Assume that (A1) and (A2) hold, then there exists a unique random equilibrium $V_1(\omega)$ of the RDS $(\theta,\varphi_1)$ such that
\begin{equation}\label{eq12}
\lim_{t\rightarrow\infty}\varphi_1(t,\theta_{-t}\omega,x)=V_1(\omega)
\end{equation}
for any $x\in\mathbb{R}^n$ and $\omega\in\Omega$. Moreover, the random equilibrium $V_1:\Omega\rightarrow\mathbb{R}^n$ is tempered.
\end{theorem}
\begin{proof}
In order to use Theorem \ref{thm1}, it is sufficient to verify (H1) and (H2).
Due to the variation-of-constants formula \cite[Theorem 3.1]{M}, it is evident that
\begin{equation}\label{eq20}
\varphi_1(t,\omega,x)=\Phi(t)x+\int_0^t\Phi(t-s)f\bigl(\varphi_1(s,\omega,x)\bigr)ds+\int_0^t\Phi(t-s)\Sigma dB(s).
\end{equation}
Thus, for any different initial values $x,y\in\mathbb{R}^n$, we deduce that
\begin{align}\label{eq21}
&|\varphi_1(t,\omega,x)-\varphi_1(t,\omega,y)|\nonumber\\
\leq&\|\Phi(t)\|\cdot|x-y|+\int_0^t\left\|\Phi(t-s)\right\|\cdot\left|f\bigl(\varphi_1(s,\omega,x)\bigr)-f\bigl(\varphi_1(s,\omega,y)\bigr)\right|ds\nonumber\\
\leq& Ce^{-\lambda t}|x-y|+LCe^{-\lambda t}\int_0^te^{\lambda s}|\varphi_1(s,\omega,x)-\varphi_1(s,\omega,y)|ds,
\end{align}
and then
\begin{equation}\label{eq22}
e^{\lambda t}|\varphi_1(t,\omega,x)-\varphi_1(t,\omega,y)|
\leq C|x-y|+LC\int_0^te^{\lambda s}|\varphi_1(s,\omega,x)-\varphi_1(s,\omega,y)|ds.
\end{equation}
Applying the Gronwall inequality, it is obvious that
\begin{equation}\label{eq23}
|\varphi_1(t,\theta_{-t}\omega,x)-\varphi_1(t,\theta_{-t}\omega,y)|
\leq C|x-y|e^{-(\lambda-LC)t}
\end{equation}
for all $x,y\in\mathbb{R}^n$, $t\geq0$ and $\omega\in\Omega$, which proves (H1).

Furthermore, set
\begin{equation}\label{OU}z_1(t,\omega)\equiv z_1(\theta_t\omega)=\int_{-\infty}^t\Phi(t-s)\Sigma dB(s)\end{equation}
for all $t\in\mathbb{R}$ and $\omega\in\Omega$, which is the Ornstein-Uhlenbeck process satisfying the following affine SDEs
\begin{equation}\label{OU2}
  dz_1(t)=Az_1(t)dt+\Sigma dB(t).
\end{equation}
By (A1), it is easy to see that the random variable $z_1(\omega)$ is tempered with respect to $\theta$ and $z_1(\theta_t\omega)$ is continuous on $\mathbb{R}$ for any $\omega\in\Omega$, see Lemma 2.5.1 in \cite{Chu} or Proposition 3.1 in \cite{CS}. Let $\varphi_{z_1}(t,\omega,x)=\varphi_1(t,\omega,x)-z_1(t,\omega)$, applying It\^{o}'s formula, it follows immediately that
\begin{equation}\label{eq13}
d[\varphi_{z_1}(t,\omega,x)]=A\varphi_{z_1}(t,\omega,x)+f\bigl(\varphi_{z_1}(t,\omega,x)+z_1(t,\omega)\bigr),
\end{equation}
and then
\begin{equation}\label{eq13-2}
\varphi_{z_1}(t,\omega,x)=\Phi(t)x+\int_0^t\Phi(t-s)f\bigl(\varphi_{z_1}(s,\omega,x)+z_1(s,\omega)\bigr)ds
\end{equation}
for all $t\geq0$ and $\omega\in\Omega$. Therefore, by (A1) and (A2), it is easily seen that
\begin{align}\label{eq14}
\left|\varphi_{z_1}(t,\omega,x)\right|
\leq& Ce^{-\lambda t}|x|+ Ce^{-\lambda t}\int_0^t e^{\lambda s}\bigl(L|z_1(\theta_s\omega)|+|f(0)|\bigr)ds\nonumber\\
&+LCe^{-\lambda t}\int_0^{t}e^{\lambda s}\left|\varphi_{z_1}(s,\omega,x)\right|ds,
\end{align}
which implies that
\begin{align}\label{eq15}
e^{\lambda t}\left|\varphi_{z_1}(t,\omega,x)\right|
\leq& C|x|+ C\int_0^t e^{\lambda s}\bigl(L|z_1(\theta_s\omega)|+|f(0)|\bigr)ds\nonumber\\
&+LC\int_0^{t}e^{\lambda s}\left|\varphi_{z_1}(s,\omega,x)\right|ds.
\end{align}
Using (\ref{eq15}) and the Gronwall inequality, we have that
\begin{equation}\label{eq16}
e^{\lambda t}\left|\varphi_{z_1}(t,\omega,x)\right|\leq C|x|e^{LCt}+C\int_0^t e^{LC(t-s)}e^{\lambda s}\bigl(L|z_1(\theta_s\omega)|+|f(0)|\bigr)ds
\end{equation}
and so
\begin{equation}\label{eq17}
\left|\varphi_{z_1}(t,\omega,x)\right|\leq C|x|e^{-(\lambda-LC)t}+C\int_0^t e^{-(\lambda-LC)(t-s)}\bigl(L|z_1(\theta_s\omega)|+|f(0)|\bigr)ds,
\end{equation}
which together with the definition of $\theta$ gives that
\begin{align}\label{eq18}
&\left|\varphi_{z_1}(t,\theta_{-t}\omega,x)\right|\nonumber\\
\leq&C|x|e^{-(\lambda-LC)t}+C\int_0^t e^{-(\lambda-LC)(t-s)}\bigl(L|z_1(\theta_{s-t}\omega)|+|f(0)|\bigr)ds\nonumber\\
=&C|x|e^{-(\lambda-LC)t}+C\int_{-t}^0 e^{(\lambda-LC)s}\bigl(L|z_1(\theta_{s}\omega)|+|f(0)|\bigr)ds\nonumber\\
\leq&C|x|+C\int_{-\infty}^0 e^{(\lambda-LC)s}\bigl(L|z_1(\theta_{s}\omega)|+|f(0)|\bigr)ds\nonumber\\
\triangleq&\widetilde R_x^1(\omega).
\end{align}
Since $LC<\lambda$ and $z_1$ is tempered, by the similar argument in \cite{JL3}, we see at once that $\widetilde R_x^1(\omega)$ is also a tempered random variable. Consequently,
\begin{align}\label{eq19}
\left|\varphi_1(t,\theta_{-t}\omega,x)-x\right|&\leq|x|+\left|\varphi_{z_1}(t,\theta_{-t}\omega,x)\right|+|z_1(\omega)|\nonumber\\
&\leq|x|+\widetilde R_x^1(\omega)+|z_1(\omega)|\nonumber\\
&\triangleq R_x^1(\omega)
\end{align}
for all $x\in\mathbb{R}^n$, $t\geq0$ and $\omega\in\Omega$. This together with Remark 1 shows that (H2) is true. The proof is complete.
\end{proof}

\subsection{\bfseries Additive white noise: One-Sided dissipative Lipschitz condition}
Secondly, we study the following $n$-dimensional SDEs with additive white noise
\begin{equation}\label{SDE2-1}
  dx(t)=g\bigl(x(t)\bigr)dt+\Sigma dB(t),
\end{equation}
where $B(t)$ and $\Sigma$ are given in (\ref{SDE1-1}), $g:\mathbb{R}^n\rightarrow\mathbb{R}^n$ is a continuously differentiable function satisfying
\begin{enumerate}[({A3})]
\item For any $x,y\in\mathbb{R}^n$,
\begin{equation}\label{eq24}
\langle x-y,g(x)-g(y)\rangle\leq -L|x-y|^2,
\end{equation}
where $L>0$ and $\langle \cdot,\cdot\rangle$ is the standard inner product in $\mathbb{R}^n$.
\end{enumerate}

\begin{enumerate}[({A4})]
\item There exist constants $a>0$, $b>0$ and $p\geq1$ such that
\begin{equation}\label{eq24-2}
|g(x)|\leq a|x|^p+b
\end{equation}
for all $x\in\mathbb{R}^n$.
\end{enumerate}

\textsc{{\it Remark}} 5.
In fact, for $n\geq2$, (A3) cannot generally be obtained from (A1). For example, set
\begin{equation}\label{eq25}
A=\left[\begin{array}{cc} 0&-2\\
 3&-1\end{array}\right],
\end{equation}
it is a simple matter to check that the eigenvalues $\lambda_{1,2}=\frac{-1\pm\sqrt{23}{\rm i}}{2}$,
\begin{equation}\label{eq26}
\Phi(t)=e^{-\frac{t}{2}}\left[\begin{array}{cc} \cos\left(\frac{\sqrt{23}}{2}t\right)+\frac{\sin\left(\frac{\sqrt{23}}{2}t\right)}{\sqrt{23}}&-\frac{4\sin\left(\frac{\sqrt{23}}{2}t\right)}{\sqrt{23}}\\
\frac{6\sin\left(\frac{\sqrt{23}}{2}t\right)}{\sqrt{23}}&\cos\left(\frac{\sqrt{23}}{2}t\right)-\frac{\sin\left(\frac{\sqrt{23}}{2}t\right)}{\sqrt{23}}\end{array}\right],
\end{equation}
and then
\begin{equation}\label{eq27}
\|\Phi(t)\|=e^{-\frac{t}{2}}\sqrt{2+\frac{8}{23}\sin^2\left(\frac{\sqrt{23}}{2}t\right)}\leq 3\sqrt{\frac{6}{23}}e^{-\frac{t}{2}}.
\end{equation}
On the other hand,
\begin{equation}\label{eq28}
\langle x,Ax\rangle=x_1x_2-x_2^2,\quad x\in\mathbb{R}^2,
\end{equation}
which implies that we can not find any $L>0$ such that
\begin{equation}\label{eq29}
\langle x,Ax\rangle\leq-L|x|^2,\quad x\in\mathbb{R}^2.
\end{equation}
That is, (A1) is true, but (A3) is false.

Set $\varphi_2(t,\omega,x)=x(t,\omega,x)$, where $x(t,\omega,x)$ is the solution of (\ref{SDE2-1}) with the initial value $x(0)=x\in\mathbb{R}^n$. Thanks to the continuous differentiability of $g$, one-sided dissipative Lipschitz condition (A3) and the Ornstein-Uhlenbeck process, we can easily show that $(\theta,\varphi_2)$ is an RDS generated by (\ref{SDE2-1}), see Theorem 2.1.1 and Corollary 2.1.1 in \cite{Chu}. Our main results in this subsection are the following theorems.

\begin{theorem}\label{thm3} Assume that (A3) and (A4) hold, then there exists a unique random equilibrium $V_2(\omega)$ of the RDS $(\theta,\varphi_2)$ such that
\begin{equation}\label{eq30}
\lim_{t\rightarrow\infty}\varphi_2(t,\theta_{-t}\omega,x)=V_2(\omega)
\end{equation}
for any $x\in\mathbb{R}^n$ and $\omega\in\Omega$. Moreover, the random equilibrium $V_2:\Omega\rightarrow\mathbb{R}^n$ is tempered.
\end{theorem}
\begin{proof}
The proof will be divided into two parts. By (\ref{SDE2-1}), it follows naturally that
\begin{equation}\label{eq37}
\varphi_2(t,\omega,x)-\varphi_2(t,\omega,y)=x-y+\int_0^t\Bigl(g\bigl(\varphi_2(s,\omega,x)\bigr)-g\bigl(\varphi_2(s,\omega,y)\bigr)\Bigr)ds
\end{equation}
which together with (A3) implies that
\begin{align}\label{eq38}
&\frac{d}{dt}|\varphi_2(t,\omega,x)-\varphi_2(t,\omega,y)|^2\nonumber\\
=&2\left\langle\varphi_2(t,\omega,x)-\varphi_2(t,\omega,y),g\bigl(\varphi_2(t,\omega,x)\bigr)-g\bigl(\varphi_2(t,\omega,y)\bigr)\right\rangle\nonumber\\
\leq&-2L|\varphi_2(t,\omega,x)-\varphi_2(t,\omega,y)|^2.
\end{align}
Combining this and the Gronwall inequality, it is obvious that
\begin{equation}\label{eq39}
|\varphi_2(t,\theta_{-t}\omega,x)-\varphi_2(t,\theta_{-t}\omega,y)|\leq e^{-L t}|x-y|
\end{equation}
for all $x,y\in\mathbb{R}^n$, $t\geq0$ and $\omega\in\Omega$, which shows (H1).

Moreover, define
\begin{equation}\label{eq31}
z_2(t,\omega)\equiv z_2(\theta_t\omega)=\int_{-\infty}^te^{-(t-s)}\Sigma dB(s),
\end{equation}
which is the stationary solution (Ornstein-Uhlenbeck process) of the following affine SDEs
\begin{equation}\label{OU3}
dz_2(t)=-z_2(t)dt+\Sigma dB(t).
\end{equation}
Using It\^{o}'s formula, we can easily have that
\begin{equation}\label{eq32}
\varphi_2(t,\omega,x)-z_2(t,\omega)=x-z_2(\omega)+\int_0^t\Bigl(g\bigl(\varphi_2(s,\omega,x)\bigr)+z_2(s,\omega)\Bigr)ds,
\end{equation}
which yields that
\begin{align}\label{eq33}
&\frac{d}{dt}|\varphi_2(t,\omega,x)-z_2(t,\omega)|^2\nonumber\\
=&2\left\langle\varphi_2(t,\omega,x)-z_2(t,\omega),g\bigl(\varphi_2(t,\omega,x)\bigr)+z_2(t,\omega)\right\rangle\nonumber\\
=&2\left\langle\varphi_2(t,\omega,x)-z_2(t,\omega),g\bigl(\varphi_2(t,\omega,x)\bigr)-g\bigl(z_2(t,\omega)\bigr)\right\rangle\nonumber\\
&+2\left\langle\varphi_2(t,\omega,x)-z_2(t,\omega),g\bigl(z_2(t,\omega)\bigr)+z_2(t,\omega)\right\rangle\nonumber\\
\leq&-2L|\varphi_2(t,\omega,x)-z_2(t,\omega)|^2+L|\varphi_2(t,\omega,x)-z_2(t,\omega)|^2\nonumber\\
&+\frac1L|g\bigl(z_2(t,\omega)\bigr)+z_2(t,\omega)|^2\nonumber\\
=&-L|\varphi_2(t,\omega,x)-z_2(t,\omega)|^2+\frac1L|g\bigl(z_2(t,\omega)\bigr)+z_2(t,\omega)|^2.
\end{align}
Hence, by the Gronwall inequality and (A4), it is clear that
\begin{align}
&|\varphi_2(t,\omega,x)-z_2(t,\omega)|^2\nonumber\\
\leq& e^{-Lt}|x-z_2(0,\omega)|^2+\frac1L\int_0^te^{-L(t-s)}|g\bigl(z_2(s,\omega)\bigr)+z_2(s,\omega)|^2ds\nonumber\\
\leq& e^{-Lt}|x-z_2(\omega)|^2+\frac2L\int_0^te^{-L(t-s)}\left(|g\bigl(z_2(s,\omega)\bigr)|^2+|z_2(s,\omega)|^2\right)ds\nonumber\\
\leq& e^{-Lt}|x-z_2(\omega)|^2+\frac2L\int_0^te^{-L(t-s)}\left(2a^2|z_2(s,\omega)|^{2p}+2b^2+|z_2(s,\omega)|^2\right)ds\nonumber\\
\leq& e^{-Lt}|x-z_2(\omega)|^2+\frac2L\int_0^te^{-L(t-s)}\left(2a^2|z_2(s,\omega)|^{2p}+2b^2+|z_2(s,\omega)|^{2p}+1\right)ds\nonumber\\
\leq& \frac{4b^2+2}{L^2}+e^{-Lt}|x-z_2(\omega)|^2+\frac{4a^2+2}{L}\int_0^te^{-L(t-s)}|z_2(s,\omega)|^{2p}ds\nonumber
\end{align}
and then
\begin{align}\label{eq35}
&|\varphi_2(t,\theta_{-t}\omega,x)-z_2(t,\theta_{-t}\omega)|^2\nonumber\\
=&|\varphi_2(t,\theta_{-t}\omega,x)-z_2(\omega)|^2\nonumber\\
\leq&\frac{4b^2+2}{L^2}+e^{-Lt}|x-z_2(\theta_{-t}\omega)|^2+\frac{4a^2+2}{L}\int_0^te^{-L(t-s)}|z_2(s,\theta_{-t}\omega)|^{2p}ds\nonumber\\
=&\frac{4b^2+2}{L^2}+e^{-Lt}|x-z_2(\theta_{-t}\omega)|^2+\frac{4a^2+2}{L}\int_{-t}^0e^{Ls}|z_2(\theta_s\omega)|^{2p}ds\nonumber\\
\leq&\frac{4b^2+2}{L^2}+\sup_{t\geq0}\left\{e^{-Lt}|x-z_2(\theta_{-t}\omega)|^2\right\}+\frac{4a^2+2}{L}\int_{-\infty}^0e^{Ls}|z_2(\theta_s\omega)|^{2p}ds\nonumber\\
\triangleq&\widetilde R_x^2(\omega).
\end{align}
Note that $z_2$ is a tempered random variable and so also is $|z_2|^p$ for any $p\geq1$, we can directly verify that the random variable $\widetilde R_x^2$ is tempered.
Accordingly,
\begin{align}\label{eq36}
|\varphi_2(t,\theta_{-t}\omega,x)-x|^2
&\leq2|x|^2+4|\varphi_2(t,\theta_{-t}\omega,x)-z_2(\omega)|^2+4|z_2(\omega)|^2\nonumber\\
&\leq2|x|^2+4\widetilde R_x^2(\omega)+4|z_2(\omega)|^2\nonumber\\
&\triangleq R_x^2(\omega),\ t\geq0,\ x\in\mathbb{R}^n,
\end{align}
which together with Remark 1 gives (H2). Using Theorem \ref{thm1}, the proof is complete.
\end{proof}

\subsection{\bfseries Multiplicative white noise: Globally Lipschitz condition}
Thirdly, we investigate the following $n$-dimensional SDEs with multiplicative white noise
\begin{equation}\label{SDE3-1}
  dx(t)=\Bigl[Ax(t)+h\bigl(x(t)\bigr)\Bigr]dt+\sum_{k=1}^{m}\sigma_kx(t)dB_k(t),
\end{equation}
where $B(t)$ and $A$ are given in (\ref{SDE1-1}), $h:\mathbb{R}^n\rightarrow\mathbb{R}^n$ and $\sigma_k=(\sigma_k^{ij})_{n\times n}$ are $n\times n$-dimensional matrices, $k=1,\ldots,m$. Let $(\theta,\Psi)$ denote the RDS generated by the corresponding linear SDEs
\begin{equation}\label{SDE3-2}
  dx(t)=Ax(t)dt+\sum_{k=1}^{m}\sigma_kx(t)dB_k(t),
\end{equation}
where $\Psi(t)=\bigl(\Psi_{ij}(t)\bigr)_{n\times n}$ is the fundamental matrix of (\ref{SDE3-2}).
To prove our main results, we need to make some assumptions on $(\theta,\Psi)$
\begin{enumerate}[(A5)]
\item
The top Lyapunov exponent for the linear RDS $(\theta,\Psi)$ is a negative real number, i.e., there exist a constant $\lambda>0$ and a
tempered random variable $\overline{R}(\omega)>0$ such that
\begin{equation}
\|\Psi(t,\omega)\|:=\left(\sum_{i=1}^n\sum_{j=1}^n|\Psi_{ij}(t,\omega)|^2\right)^{\frac12}\leq \overline R(\omega)e^{-\lambda t}
\end{equation}
holds for all $t\geq0$, $\omega\in\Omega$. In addition,
$\overline R\in\mathcal{L}^1(\Omega,\mathscr{F},\mathbb{P};\mathbb{R}_+)$ and
\[\|\overline R\|_{\mathcal{L}^1}=\mathbb{E}\overline R=\int_\Omega \overline R(\omega)\mathbb{P}(d\omega).\]
\end{enumerate}

\begin{enumerate}[(A6)]
\item $h$ is globally Lipschitz continuous, i.e.,
\begin{equation}\label{eq11-2}
|h(x)-h(y)|\leq L|x-y|
\end{equation}
for all $x,y\in\mathbb{R}^n$, where $L>0$ is the Lipschitz constant satisfying $\frac{L\|\overline R\|_{\mathcal{L}^1}}{\lambda}<1$
\end{enumerate}

Let $\varphi_3(t,\omega,x)$ be the RDS generated by the SDEs (\ref{SDE3-1}), where $\varphi_3(t,\omega,x)=x(t,\omega,x)$ stands for the solution of (\ref{SDE3-1}) with the initial value $x(0)=x\in\mathbb{R}^n$. The following theorem describes the existence and global stability of random equilibria for $(\theta,\varphi_3)$.

\begin{theorem}\label{thm4} Assume that (A5) and (A6) hold, then there exists a unique random equilibrium $V_3(\omega)$ of the RDS $(\theta,\varphi_3)$ such that
\begin{equation}\label{eq40}
\lim_{t\rightarrow\infty}\varphi_3(t,\theta_{-t}\omega,x)=V_3(\omega)
\end{equation}
for any $x\in\mathbb{R}^n$ and $\omega\in\Omega$. Moreover, the random equilibrium $V_3:\Omega\rightarrow\mathbb{R}^n$ is tempered.
\end{theorem}
\begin{proof}
Combining the variation-of-constants formula \cite[Chapter 3, Theorem 3.1]{M} and the cocycle property of $\Psi$, it is easily seen that
\begin{align}\label{eq41}
\varphi_3(t,\omega,x)&=\Psi(t,\omega)x+\Psi(t,\omega)\int_0^t\Psi^{-1}(s,\omega)h\bigl(\varphi_3(s,\omega,x)\bigr)ds\nonumber\\
&=\Psi(t,\omega)x+\int_0^t\Psi(t-s,\theta_s\omega)h\bigl(\varphi_3(s,\omega,x)\bigr)ds,
\end{align}
which together with (A5) and (A6) deduces that for any $x,y\in\mathbb{R}^n$,
\begin{align}\label{eq42}
&|\varphi_3(t,\omega,x)-\varphi_3(t,\omega,y)|\nonumber\\
\leq&\overline R(\omega)e^{-\lambda t}|x-y|+\int_0^t\overline R(\theta_s\omega)e^{-\lambda(t-s)}L|\varphi_3(s,\omega,x)-\varphi_3(s,\omega,y)|ds
\end{align}
and hence
\begin{align}\label{eq43}
&e^{\lambda t}|\varphi_3(t,\omega,x)-\varphi_3(t,\omega,y)|\nonumber\\
\leq&\overline R(\omega)|x-y|+L\int_0^t\overline R(\theta_s\omega)e^{\lambda s}|\varphi_3(s,\omega,x)-\varphi_3(s,\omega,y)|ds.
\end{align}
Applying the Gronwall inequality, we have that
\begin{equation}\label{eq44}
e^{\lambda t}|\varphi_3(t,\omega,x)-\varphi_3(t,\omega,y)|
\leq\overline R(\omega)|x-y|\exp\left(L\int_0^t\overline R(\theta_s\omega)ds\right),
\end{equation}
which implies that
{\small\begin{align}\label{eq45}
&|\varphi_3(t,\theta_{-t}\omega,x)-\varphi_3(t,\theta_{-t}\omega,y)|\nonumber\\
\leq&\overline R(\theta_{-t}\omega)|x-y|\exp\left(-\lambda t+L\int_0^t\overline R(\theta_{s-t}\omega)ds\right)\nonumber\\
\leq&e^{-\varepsilon_0t}\overline R(\theta_{-t}\omega)\exp\left(-(\lambda-2\varepsilon_0)t+L\int_{-t}^0\overline R(\theta_{s}\omega)ds\right)e^{-\varepsilon_0t}|x-y|\nonumber\\
\leq&\sup_{t\geq0}\left\{e^{-\varepsilon_0t}\overline R(\theta_{-t}\omega)\right\}\cdot\sup_{t\geq0}\left\{\exp\left(-(\lambda-2\varepsilon_0)t+L\int_{-t}^0\overline R(\theta_{s}\omega)ds\right)\right\}e^{-\varepsilon_0t}|x-y|\nonumber\\
\triangleq&\widehat{R}_1(\omega)e^{-\varepsilon_0t}|x-y|,
\end{align}}
where $0<\varepsilon_0<\frac{\lambda-L\|\overline R\|_{\mathcal{L}^1}}{2}$. Here, due to the Birkhoff-Khintchin ergodic theorem (see Arnold \cite[Appendix]{A}) and (A6), it follows that
\begin{equation}\label{eq46}
\lim_{t\rightarrow\infty}\exp\left(-(\lambda-2\varepsilon_0)t+L\int_{-t}^0\overline R(\theta_s\omega)ds\right)=0,\quad \omega\in\Omega,
\end{equation}
which together with the temperedness of $\overline R$ yields that the random variable $\widehat{R}_1$ is well defined. Therefore, the condition (H1) is correct.

In addition, using (\ref{eq41}), (A5) and (A6), we can see that
\begin{equation}\label{eq47}
|\varphi_3(t,\omega,x)|\leq\overline R(\omega)e^{-\lambda t}|x|+\int_0^t\overline R(\theta_s\omega)e^{-\lambda(t-s)}\bigl[L|\varphi_3(s,\omega,x)|+|h(0)|\bigr]ds
\end{equation}
and thus
\begin{align}\label{eq48}
&e^{\lambda t}|\varphi_3(t,\omega,x)|\nonumber\\
\leq&\overline R(\omega)|x|+L\int_0^t\overline R(\theta_s\omega)e^{\lambda s}|\varphi_3(s,\omega,x)|ds+\int_0^t\overline R(\theta_s\omega)e^{\lambda s}|h(0)|ds.
\end{align}
Again thanks to the Gronwall inequality, it is clear that
\begin{equation}\label{eq49}
e^{\lambda t}|\varphi_3(t,\omega,x)|\leq\overline R(\omega)|x|e^{L\int_0^t\overline R(\theta_s\omega)ds}
+\int_0^t\overline R(\theta_s\omega)e^{\lambda s}|h(0)|e^{L\int_s^t\overline R(\theta_u\omega)du}ds
\end{equation}
and so
\begin{align}
&|\varphi_3(t,\theta_{-t}\omega,x)|\nonumber\\
\leq&\overline R(\theta_{-t}\omega)e^{-\lambda t+L\int_0^t\overline R(\theta_{s-t}\omega)ds}|x|
+\int_0^t\overline R(\theta_{s-t}\omega)e^{-\lambda(t-s)}|h(0)|e^{L\int_s^t\overline R(\theta_{u-t}\omega)du}ds\nonumber\\
=&\overline R(\theta_{-t}\omega)e^{-\lambda t+L\int_{-t}^0\overline R(\theta_s\omega)ds}|x|
+\int_0^t\overline R(\theta_{s-t}\omega)e^{-\lambda(t-s)}|h(0)|e^{L\int_{s-t}^0\overline R(\theta_{u}\omega)du}ds\nonumber\\
=&\overline R(\theta_{-t}\omega)e^{-\lambda t+L\int_{-t}^0\overline R(\theta_s\omega)ds}|x|
+\int_{-t}^0\overline R(\theta_s\omega)e^{\lambda s}|h(0)|e^{L\int_s^0\overline R(\theta_{u}\omega)du}ds\nonumber\\
\leq&\sup_{t\geq0}\left\{e^{-\varepsilon_0t}\overline R(\theta_{-t}\omega)\right\}\cdot\sup_{t\geq0}\left\{\exp\left(-(\lambda-\varepsilon_0)t+L\int_{-t}^0\overline R(\theta_s\omega)ds\right)\right\}|x|\nonumber\\
&+\int_{-\infty}^0\overline R(\theta_s\omega)e^{\lambda s}|h(0)|e^{L\int_s^0\overline R(\theta_{u}\omega)du}ds\nonumber\\
\triangleq&\widetilde R_x^3(\omega)+\widetilde R^4(\omega)\nonumber
\end{align}
for all $t\geq0$ and $\omega\in\Omega$, where $0<\varepsilon_0<\frac{\lambda-L\|\overline R\|_{\mathcal{L}^1}}{2}$.

Next, we will show that $\widetilde R_x^3$ and $\widetilde R^4$ are both tempered. Note that $\overline R$ is tempered, in order to prove the temperedness of $\widetilde R_x^3$, it is sufficient to show that $\widetilde R^5$ is tempered, where
\begin{equation}\label{eq51}
\widetilde R^5(\omega)=\sup_{t\geq0}\left\{\exp\left(-(\lambda-\varepsilon_0)t+L\int_{-t}^0\overline R(\theta_s\omega)ds\right)\right\}.
\end{equation}
For any $\gamma>0$ and $\omega\in\Omega$, using the Birkhoff-Khintchin ergodic theorem (see Arnold \cite[Appendix]{A}), choose $\varepsilon_1<\min\{\frac{\gamma}{4L},\frac{\lambda-\varepsilon_0-L\|\overline R\|_{\mathcal{L}^1}}{L}\}$, then there exists $T=T(\omega,\varepsilon_1)>0$ such that
\begin{equation}
\bigl(\|\overline R\|_{\mathcal{L}^1}-\varepsilon_1\bigr)t\leq\int_{-t}^0\overline R(\theta_u\omega)du\leq\bigl(\|\overline R\|_{\mathcal{L}^1}+\varepsilon_1\bigr)t,\quad t\geq T,\nonumber
\end{equation}
which yields that
\begin{align}\label{eq52}
&\sup_{t\geq0}\bigl\{e^{-\gamma t}\widetilde R^5(\theta_{-t}\omega)\bigr\}\nonumber\\
=&\sup_{t\geq0}\left\{e^{-\gamma t}\sup_{s\geq0}\left\{\exp\left(-(\lambda-\varepsilon_0)s+L\int_{-s}^0\overline R(\theta_u\circ\theta_{-t}\omega)du\right)\right\}\right\}\nonumber\\
=&\sup_{t\geq0}\left\{e^{-\gamma t}\sup_{s\geq0}\left\{\exp\left(-(\lambda-\varepsilon_0)s+L\int_{-s-t}^{-t}\overline R(\theta_u\omega)du\right)\right\}\right\}\nonumber\\
\leq&\sup_{0\leq t\leq T}\left\{e^{-\gamma t}\sup_{s\geq0}\left\{\exp\left(-(\lambda-\varepsilon_0)s+L\int_{-s-t}^{-t}\overline R(\theta_u\omega)du\right)\right\}\right\}\nonumber\\
&+\sup_{t\geq T}\left\{e^{-\gamma t}\sup_{s\geq0}\left\{\exp\left(-(\lambda-\varepsilon_0)s+L\int_{-s-t}^{-t}\overline R(\theta_u\omega)du\right)\right\}\right\}\nonumber\\
\leq&\sup_{0\leq t\leq T}\left\{e^{(\lambda-\varepsilon_0-\gamma)t}\sup_{s\geq0}\left\{\exp\left(-(\lambda-\varepsilon_0)(s+t)+L\int_{-s-t}^{0}\overline R(\theta_u\omega)du\right)\right\}\right\}\nonumber\\
&+\sup_{t\geq T}\left\{e^{-\gamma t}\sup_{s\geq0}\left\{\exp\left(-(\lambda-\varepsilon_0-L\|\overline R\|_{\mathcal{L}^1}-L\varepsilon_1)s+2L\varepsilon_1t\right)\right\}\right\}\nonumber\\
\leq&e^{(\lambda-\varepsilon_0)T}\sup_{s\geq0}\left\{\exp\left(-(\lambda-\varepsilon_0)s+L\int_{-s}^{0}\overline R(\theta_u\omega)du\right)\right\}+e^{-(\gamma-2L\varepsilon_1)T}\nonumber\\
<&\infty.
\end{align}
That is, $\widetilde R_x^3$ is tempered. Furthermore, note that
\begin{align}\label{eq53}
&\int_{-\infty}^0\overline R(\theta_s\omega)e^{\lambda s}|h(0)|e^{L\int_s^0\overline R(\theta_{u}\omega)du}ds\nonumber\\
\leq&\sup_{s\leq0}\left\{\exp\left((\lambda-\varepsilon_0)s+L\int_{s}^0\overline R(\theta_u\omega)du\right)\right\}\int_{-\infty}^0\overline R(\theta_s\omega)e^{\varepsilon_0s}|h(0)|ds\nonumber\\
=&\widetilde R^5(\omega)\int_{-\infty}^0\overline R(\theta_s\omega)e^{\varepsilon_0s}|h(0)|ds,
\end{align}
which together with the temperedness of $\widetilde R^5$ and $\overline R$ implies that $\widetilde R^4$ is also tempered. Let $R_x^3=\widetilde R_x^3+\widetilde R^4$, it follows immediately that $R_x^3$ is a tempered random variable. Combining this and Remark 1, it is easy to verify the condition (H2). From Theorem \ref{thm1}, the proof is complete.
\end{proof}

\subsection{\bfseries Multiplicative white noise: One-Sided dissipative Lipschitz condition}
Fourthly, for convenience, we will discuss the following $n$-dimensional Stratonovich SDEs with multiplicative white noise
\begin{equation}\label{SDE4-1}
  dx_i(t)=g_i\bigl(x(t)\bigr)dt+\sum_{k=1}^{m}c_{k}x_i(t)\circ dB_k(t),\quad i=1,\ldots,n,
\end{equation}
where $c_{k}$ for $k=1,\ldots,m$ are constants, $B(t)$ is defined in (\ref{SDE1-1}) and $g=(g_1,\ldots,g_n):\mathbb{R}^n\rightarrow\mathbb{R}^n$ is given in (\ref{SDE2-1}) satisfying (A3).

Define $\varphi_4(t,\omega,x)=x(t,\omega,x)$, where $x(t,\omega,x)$ is the solution of (\ref{SDE4-1}) with the initial value $x(0)=x\in\mathbb{R}^n$. In what follows, we will denote by $u(\omega)$ the random variable in $\mathbb{R}^m$ such that $u(t,\omega)\triangleq u(\theta_t\omega)=\bigl(u_1(\theta_t\omega),\ldots,u_m(\theta_t\omega)\bigr)$ is the stationary Ornstein-Uhlenbeck process which solves the equations
\begin{equation}\label{OU4-1}
du_k(t)=-u_k(t)dt+dB_k(t),\quad k=1,\ldots,m.
\end{equation}
In order to verify the existence of stochastic flows, set $y(t,\omega,y)\triangleq x(t,\omega,x)\cdot\exp\{-u^c(\theta_t\omega)\}$, where
\begin{equation}\label{OU4-2}
u^c(\omega)=\sum_{k=1}^mc_{k}u_k(\omega)
\end{equation}
is a tempered random variable satisfying
\begin{equation}\label{OU4-3}
\lim_{t\rightarrow\pm\infty}\frac{|u^c(\theta_t\omega)|}{|t|}=0,\quad \omega\in\Omega.
\end{equation}
Using It\^{o}'s formula and  (\ref{SDE4-1}), we have that
\begin{equation}\label{SDE4-2}
  \frac{dy(t)}{dt}=G\bigl(\theta_t\omega,y(t)\bigr),
\end{equation}
where
\begin{equation}\label{SDE4-3}
G\bigl(\omega,y\bigr)=\exp\{-u^c(\omega)\}\cdot g\bigl(y\cdot\exp\{u^c(\omega)\}\bigr)+y\cdot u^c(\omega).
\end{equation}
Thanks to Theorem 2.1.1 and Corollary 2.1.1 in \cite{Chu}, it follows that the RDEs (\ref{SDE4-2}) generates an RDS $(\theta,\psi)$ in $\mathbb{R}^n$.
Moreover, we can get the relation
\begin{equation}\label{SDE4-4}
\varphi_4(t,\omega,x)=T\Bigl(\theta_t\omega,\psi\bigl(t,\omega,T^{-1}(\omega,x)\bigr)\Bigr),\quad t\geq0,\ \omega\in\Omega,
\end{equation}
where the linear mapping $T(\omega,\cdot)$ is a homeomorphism on $\mathbb{R}^n$ given by
\begin{equation}\label{SDE4-5}
T(\omega,y)=y\cdot\exp\{u^c(\omega)\},\quad \omega\in\Omega.
\end{equation}
Applying Lemma 2.2 in \cite{CKS}, it is clear that $(\theta,\varphi_4)$ is an RDS generated by (\ref{SDE4-1}). Now, we can show the existence of globally stable stationary solutions for $(\theta,\varphi_4)$ in the following theorems.

\begin{theorem}\label{thm5} Assume that (A3) holds, then there exists a unique random equilibrium $V_4(\omega)$ of the RDS $(\theta,\varphi_4)$ such that
\begin{equation}\label{SDE4-6}
\lim_{t\rightarrow\infty}\varphi_4(t,\theta_{-t}\omega,x)=V_4(\omega)
\end{equation}
for any $x\in\mathbb{R}^n$ and $\omega\in\Omega$. Moreover, the random equilibrium $V_4:\Omega\rightarrow\mathbb{R}^n$ is tempered.
\end{theorem}
\begin{proof}
By Theorem \ref{thm1}, we only need to show (H1) and (H2). Firstly, from (\ref{SDE4-1}) and It\^{o}'s formula, we have that
\begin{align}\label{SDE4-7}
&d\left|\varphi_4(t,\omega,x)-\varphi_4(t,\omega,y)\right|^2\nonumber\\
=&2\Bigl\langle\varphi_4(t,\omega,x)-\varphi_4(t,\omega,y),g\bigl(\varphi_4(t,\omega,x)\bigr)-g\bigl(\varphi_4(t,\omega,y)\bigr)\Bigr\rangle dt\nonumber\\
&+2\left|\varphi_4(t,\omega,x)-\varphi_4(t,\omega,y)\right|^2\circ\sum_{k=1}^mc_{k}dB_k(t)\nonumber\\
\leq&-2L\left|\varphi_4(t,\omega,x)-\varphi_4(t,\omega,y)\right|^2dt\nonumber\\
&+2\sqrt{\sum_{k=1}^mc_{k}^2}\cdot\left|\varphi_4(t,\omega,x)-\varphi_4(t,\omega,y)\right|^2\circ d\widetilde{B}(t),
\end{align}
where
\begin{equation}\label{SDE4-8}
\widetilde{B}(t)=\frac{1}{\sqrt{\sum_{k=1}^mc_{k}^2}}\sum_{k=1}^mc_{k}B_k(t),\quad t\geq0,
\end{equation}
is a one dimensional two-sided Brownian motion, see Theorem 1.4.4 in \cite{M}. Combining (\ref{SDE4-7}) and the comparison theorem for solutions of SDE, see Theorem 1.1 in \cite{IW}, it is evident that
\begin{equation}\label{SDE4-9}
\left|\varphi_4(t,\omega,x)-\varphi_4(t,\omega,y)\right|^2\leq|x-y|^2\exp\left\{-2Lt+2\sqrt{\sum_{k=1}^mc_{k}^2}\widetilde{B}(t,\omega)\right\}
\end{equation}
for all $t\geq0$ and $\omega\in\Omega$. Let $C=\sqrt{\sum_{k=1}^mc_{k}^2}$, it is a simple matter that
\begin{align}\label{SDE4-10}
&\left|\varphi_4(t,\theta_{-t}\omega,x)-\varphi_4(t,\theta_{-t}\omega,y)\right|\nonumber\\
\leq&|x-y|\exp\left(-Lt-C\widetilde{B}(-t,\omega)\right)\nonumber\\
\leq&\sup_{t\geq0}\left\{\exp\left(-\varepsilon_2t-C\widetilde{B}(-t,\omega)\right)\right\}e^{-(L-\varepsilon_2)t}|x-y|\nonumber\\
\triangleq&\widehat{R}_2(\omega)e^{-(L-\varepsilon_2)t}|x-y|,
\end{align}
where $0<\varepsilon_2<L$ and the random variable $\widehat{R}_2$ is well defined based on the law of iterated logarithm. Therefore, we have proved the condition (H1).

In the other hand, by (\ref{SDE4-2}) and (\ref{SDE4-3}), we can obtain that
\begin{align}\label{SDE4-11}
&\frac{d}{dt}|\psi(t,\omega,y)|^2\nonumber\\
=&2\Bigl\langle\psi(t,\omega,y),G\bigl(\theta_t\omega,\psi(t,\omega,y)\bigr)\Bigr\rangle\nonumber\\
=&2\Bigl\langle\psi(t,\omega,y),\exp\{-u^c(\theta_t\omega)\}\cdot g\bigl(\psi(t,\omega,y)\cdot\exp\{u^c(\theta_t\omega)\}\bigr)\Bigr\rangle\nonumber\\
&+2u^c(\theta_t\omega)\cdot|\psi(t,\omega,y)|^2\nonumber\\
=&2\exp\{-2u^c(\theta_t\omega)\}\Bigl\langle\psi(t,\omega,y)\cdot\exp\{u^c(\theta_t\omega)\},g\bigl(\psi(t,\omega,y)\cdot\exp\{u^c(\theta_t\omega)\}\bigr)\Bigr\rangle\nonumber\\
&+2u^c(\theta_t\omega)\cdot|\psi(t,\omega,y)|^2\nonumber\\
\leq&-2L\exp\{-2u^c(\theta_t\omega)\}\left|\psi(t,\omega,y)\cdot\exp\{u^c(\theta_t\omega)\}\right|^2\nonumber\\
&+2\exp\{-2u^c(\theta_t\omega)\}\Bigl\langle\psi(t,\omega,y)\cdot\exp\{u^c(\theta_t\omega)\},g(0)\Bigr\rangle\nonumber\\
&+2u^c(\theta_t\omega)\cdot|\psi(t,\omega,y)|^2\nonumber\\
\leq&-2L\left|\psi(t,\omega,y)\right|^2+L\left|\psi(t,\omega,y)\right|^2+\frac{1}{L}\exp\{-2u^c(\theta_t\omega)\}\cdot|g(0)|^2\nonumber\\
&+2u^c(\theta_t\omega)\cdot|\psi(t,\omega,y)|^2\nonumber\\
=&\bigl(-L+2u^c(\theta_t\omega)\bigr)\left|\psi(t,\omega,y)\right|^2+\frac{|g(0)|^2}{L}\exp\{-2u^c(\theta_t\omega)\},
\end{align}
which together with the Gronwall inequality implies that
\begin{align}\label{SDE4-12}
&|\psi(t,\omega,y)|^2\nonumber\\
\leq&e^{-Lt+2\int_0^tu^c(\theta_s\omega)ds}|y|^2
+\frac{|g(0)|^2}{L}\int_0^te^{-2u^c(\theta_s\omega)}e^{-L(t-s)+2\int_s^tu^c(\theta_\tau\omega)d\tau}ds.
\end{align}
From (\ref{SDE4-4}) and (\ref{SDE4-5}), it is obvious that
\begin{equation}\label{SDE4-13}
\varphi_4(t,\omega,x)=\exp\{u^c(\theta_t\omega)\}\cdot\psi\bigl(t,\omega,\exp\{-u^c(\omega)\}x\bigr),\quad t\geq0,\ \omega\in\Omega,
\end{equation}
and then
\begin{align}\label{SDE4-14}
&|\varphi_4(t,\omega,x)|^2\nonumber\\
\leq&e^{-Lt+2u^c(\theta_t\omega)-2u^c(\omega)+2\int_0^tu^c(\theta_s\omega)ds}|x|^2\nonumber\\
&+\frac{|g(0)|^2}{L}e^{2u^c(\theta_t\omega)}\int_0^te^{-2u^c(\theta_s\omega)}e^{-L(t-s)+2\int_s^tu^c(\theta_\tau\omega)d\tau}ds.
\end{align}
Therefore,
\begin{align}\label{SDE4-15}
&|\varphi_4(t,\theta_{-t}\omega,x)|^2\nonumber\\
\leq&e^{-Lt+2u^c(\omega)-2u^c(\theta_{-t}\omega)+2\int_0^tu^c(\theta_{s-t}\omega)ds}|x|^2\nonumber\\
&+\frac{|g(0)|^2}{L}e^{2u^c(\omega)}\int_0^te^{-2u^c(\theta_{s-t}\omega)}e^{-L(t-s)+2\int_s^tu^c(\theta_{\tau-t}\omega)d\tau}ds\nonumber\\
=&e^{-Lt+2u^c(\omega)-2u^c(\theta_{-t}\omega)+2\int_{-t}^0u^c(\theta_{s}\omega)ds}|x|^2\nonumber\\
&+\frac{|g(0)|^2}{L}e^{2u^c(\omega)}\int_{-t}^0e^{-2u^c(\theta_{s}\omega)+Ls+2\int_s^0u^c(\theta_{\tau}\omega)d\tau}ds.
\end{align}
By the Birkhoff-Khintchin ergodic theorem (see Arnold \cite[Appendix]{A}), we can easily show that
\begin{equation}\label{SDE4-16}\lim_{t\rightarrow\infty}\frac1t\int_{-t}^0u^c(\theta_s\omega)ds=\mathbb{E}u^c=0,\quad \omega\in\Omega,\end{equation}
which together with (\ref{OU4-3}) gives that
\begin{align}\label{SDE4-17}
&|\varphi_4(t,\theta_{-t}\omega,x)|^2\nonumber\\
\leq&\sup_{t\geq0}\left\{\exp\left(-Lt-2u^c(\theta_{-t}\omega)+2\int_{-t}^0u^c(\theta_{s}\omega)ds\right)\right\}e^{2u^c(\omega)}|x|^2\nonumber\\
&+\frac{|g(0)|^2}{L}e^{2u^c(\omega)}\sup_{s\leq0}\left\{\exp\left(\frac L2s-2u^c(\theta_{s}\omega)+2\int_{s}^0u^c(\theta_{\tau}\omega)d\tau\right)\right\}
\int_{-\infty}^0e^{\frac L2s}ds\nonumber\\
\leq&\sup_{t\geq0}\left\{\exp\left(-\frac L2t-2u^c(\theta_{-t}\omega)+2\int_{-t}^0u^c(\theta_{s}\omega)ds\right)\right\}e^{2u^c(\omega)}\left(|x|^2+\frac{2|g(0)|^2}{L^2}\right)\nonumber\\
\triangleq&R_x^4(\omega)
\end{align}
for all $t\geq0$ and $\omega\in\Omega$. From (\ref{OU4-3}) and (\ref{SDE4-16}), we check at once that the random variable $R_x^4$ is tempered. Consequently, by (\ref{SDE4-17}), Remark 1 and Theorem \ref{thm1}, the proof is complete.
\end{proof}

\textsc{{\it Remark}} 6.
Compared to the classical results, see Theorem 4.4 and Corollary 4.7 in \cite{CKS}, we remove the globally Lipschitz condition and the commutativity between the function $g$ and the transform $T$ in Theorem \ref{thm5}.

\textsc{{\it Remark}} 7.
{\rm If $h(0)=0$ and $g(0)=0$, then $V_3(\omega)=V_4(\omega)\equiv0$ for all $\omega\in\Omega$. If $h(0)\neq0$ and $g(0)\neq0$, it is evident that the random equilibria $V_3$ and $V_4$ presented in Theorem \ref{thm4} and Theorem \ref{thm5} are nontrivial.

\section{Conclusion and open problems}
In this paper, under some suitable conditions, we have proved an abstract criterion to guarantee that the RDS $(\theta,\varphi)$ admits a unique stationary solution, which is exponential stable and attracts all pull-back trajectories of $(\theta,\varphi)$, see Theorem \ref{thm1}. In fact, the conditions (H1) and (H2) given in Theorem \ref{thm1} are highly concise and readily verifiable, which can be applied to various stochastic systems, see Theorem \ref{thm2}-Theorem \ref{thm5} in Section 3. For simplicity, we only consider the application in the study of SDEs in Section 3. In fact, the method presented in this paper may be applied to investigate the dynamical behaviour of SFDEs, SPDEs and so on. We leave it as an open problem. For example, whether the conditions (H1) and (H2) are satisfied for nonlinear SFDEs under a one-sided Lipschitz condition, i.e.
\[\bigl\langle\xi(0)-\eta(0),F(\xi)-F(\eta)\bigr\rangle\leq-\lambda_1|\xi(0)-\eta(0)|^2+\lambda_2\int_{-\tau}^0|\xi(s)-\eta(s)|^2\mu(ds)\]
for $\xi,\eta\in C_\tau$, where $C_\tau:=C\bigl([-\tau,0],\mathbb{R}^n\bigr)$ denotes the Banach space of continuous functions, $F:C_\tau\rightarrow\mathbb{R}^n$, $\lambda_1>\lambda_2>0$ and $\mu$ is a probability measure on $[-\tau,0]$. Moreover, the synchronization and global stability for random mappings can also be studied. These important problems will be the subject of subsequent research.

\end{document}